\documentclass[12pt]{amsart}
\usepackage{amssymb,amscd,verbatim,
mathrsfs, latexsym,amsmath,amsthm}
\usepackage{a4wide}
\usepackage{color}

\numberwithin{equation}{section}
\newtheorem{theorem}{Theorem}[section]
\newtheorem{proposition}[theorem]{Proposition}

\newtheorem{corollary}[theorem]{Corollary}
\theoremstyle{definition}
\newtheorem{definition}[theorem]{Definition}

\newtheorem{example}[theorem]{Example}
\newtheorem{prob}[theorem]{Problem}

\newtheorem{conjecture}[theorem]{Conjecture}

\def\tr{\text{tr}}

\DeclareMathOperator{\sgn}{sgn}

\def\sR{\hbox{I\kern-.1667em\hbox{R}}}

\newcommand{\R}{\mathbb R}

\newcommand{\C}{\mathbb C}

\newcommand{\Z}{\mathbb Z}

\newcommand{\ve}{\varepsilon}

\def\tr{\hbox{Tr}}

\def\tr{\mathrm{tr}}

\begin{document}
\title[Siegel-Whittaker functions on ${\rm Sp}(2,{\R})$]
{An explicit integral representation of Siegel-Whittaker functions on ${\rm Sp}(2,{\R})$ 
for the large discrete series representations}

\author[Y. Gon]{Yasuro Gon}
\author[T. Oda]{Takayuki Oda}

\address{Faculty of Mathematics\\ Kyushu University\\
744 Motooka\\ Nishi-ku\\ Fukuoka 819-0395\\ Japan}
\email{ygon@math.kyushu-u.ac.jp}

\address{Graduate School of Mathematical Sciences\\ The University of Tokyo\\
3-8-1 Komaba\\ Meguro-ku\\ Tokyo 153-8914\\ Japan}
\email{takayuki@ms.u-tokyo.ac.jp}

\thanks{
T. O. was partially supported by JSPS Grant-in-Aid for Scientific Research (A) no. 23244003.
Y. G. was partially supported by JSPS Grant-in-Aid for Scientific Research (C) no. 26400017.
}

\date{\today}

\dedicatory{In Memoriam: Fumiyuki Momose}

\begin{abstract}
{We obtain an explicit integral representation of Siegel-Whittaker functions on ${\rm Sp}(2,{\R})$ 
for the large discrete series representations. 
We have another integral expression different from that of Miyazaki \cite{Mi}.}
\end{abstract}


\maketitle


\section{Introduction}

In this article, we study Siegel-Whittaker functions on 
$G={\rm Sp}(2,{\R})$, the real symplectic group of degree
two for the large discrete series representations. 
Let $P_S$ be the Siegel parabolic subgroup of $G$, which is a maximal 
parabolic subgroup with abelian unipotent radical $N_S$.  
Let $\pi$ be an admissible representation of $G$ and $\xi$ 
be a {\it definite} unitary character
of $N_S$. Siegel-Whittaker model for an admissible $\pi$ is a realization of $\pi$
in the induced module from a certain closed subgroup $R$ which contains $N_S$. 
(See (\ref{R}) for definition of $R$.)  
We consider the intertwining space
\[   SW(\pi;\eta) 
= \mathrm{Hom}_{(\mathfrak{g}_{\C}, K)}(\pi, C_{\eta}^{\infty}(R \backslash G)), 
\]
where $\eta$ is an irreducible $R$-module such that $\eta|_{N_S}$ contains $\xi$,
$K$ is an maximal compact subgroup of $G$ and $\mathfrak{g}_{\C}$
is the complexification of the Lie algebra of $G$.
A function in this realization is called a {\it Siegel-Whittaker function}
for $\pi$. 

Takuya Miyazaki obtained a system of partial differential equations satisfied by Siegel-Whittaker
functions for the large discrete representations in \cite{Mi}. He also obtained multiplicity
one property and its formal power series solutions. 
In this article, we investigate further the system obtained by Miyazaki 
(Proposition \ref{Miyazaki})
and give an explicit integral representation of the Siegel-Whittaker functions, 
which is of rapid decay,   
for a large discrete series representation $\pi$. (Theorem \ref{th1}.)
In other words, we show that the rapidly decreasing Siegel-Whittaker function 
for the large
discrete series representation $\pi$ is uniquely determined and (up to polynomials) 
described by the
{\it partially confluent hypergeometric functions} in the $A$-radial part $(a_1,a_2) \in (\R_{>0})^2$:
\begin{equation} \label{SW}
\Psi_{\alpha,\beta}(a_1,a_2)
   = e^{-2 \pi (h_1a_1^2+h_2a_2^2)} 
   \int_{0}^{1}  F_{SW} \bigl( 2 \pi h_1 a_1^2 t +2 \pi h_2 a_2^2(1-t) \bigr) \, t^{\alpha-1}
(1-t)^{\beta-1} \, dt, 
\end{equation}
with 
\[ F_{SW}(x) = F(x) = e^x x^{-\gamma}
W_{\kappa,\mu}(2x).
\]    
Here, $W_{\kappa,\mu}(x)$ is the Whittaker's classical confluent hypergeometric function, 
$\alpha, \beta, \gamma>0$ and $\kappa, \mu$ are half integers
depending on parameters of representations $\pi$ and $\eta$. 
This kind of integral already appeared in Gon \cite{Go}, treated the cases
for the large discrete series representations and $P_{J}$-principal series 
representations of $\mathrm{SU}(2,2)$.  

The reason we begin to believe the validity of similar formula for ${\rm Sp}(2,{\R})$, 
is the paper Hirano-Ishii-Oda \cite{HIO}.
By this new integral expression (\ref{SW}), 
it seems possible to extend the former argument in \cite{HIO} 
of the confluence from Siegel-Whittaker functions to Whittaker functions 
for $P_J$-principal series, to the confluence 
from Siegel-Whittaker functions to Whittaker functions \cite{O1}
for the large discrete series of  ${\rm Sp}(2,{\R})$.
Here we recall the template of the integral formulas of Whittaker functions  in $(a_1,a_2) \in (\R_{>0})^2$:
\begin{equation} \label{W}
\mu_{W}(a_1, a_2) \int_{0}^{\infty} F_{W}(t) \, t^{-C} e^{- A \, t^2/a_2^2 - B \, a_1^2/t^2} \, dt, 
\end{equation}
where $F_W(t)=W_{0,\gamma}(t)$ with a positive integral parameter $\gamma$, 
$\mu_{W}(a_1,a_2)$ is a monomial in $a_1, a_2$ times $e^{\delta a_2^2}$ ($\delta>0$) and
$A, B, C$  are positive real parameters.  

Probably we can proceed a bit more.  Since Iida \cite{Iida}, Theorems 8.7 and 8.9 
(Formulas (8.8) and (8.10), resp.) give analogous integral expressions 
for the matrix coefficient of  $P_J$-principal series, 
and also Oda \cite{O2} (p.247), Theorem 6.2 similarly gives analogous integral expression for matrix coefficients of the large discrete series. 
We notice that the template of formulas in Theorem 6.2 of \cite{O2} is given by
\begin{equation} \label{MC}
\mu_{MC}(a_1, a_2) \int_{0}^{1} F_{MC} \bigl(t x_1 +(1-t) x_2 \bigr) \,  
t^{\alpha-1}(1-t)^{\beta-1} \, dt,
\end{equation}
with $x_i = - (a_i-a_i^{-1})^2/4$ $(i=1,2)$. Here $F_{MC}(x)={}_{2}F_{1}( A, B ; C : x)$
with $A, B, C, \alpha, \beta$ determined by Harish-Chandra parameters
and the components of the minimal $K$-type. The elementary factor $\mu_{MC}(a_1,a_2)$ is a monomial
of $\bigl( (a_i \pm a_i^{-1})/2 \bigr)^{\pm 1/2}$ $(i=1,2)$.

These facts suggests that there will be a deformation of confluence 
(\ref{MC}) $\mapsto$ (\ref{SW})
from the matrix coefficients realization to Siegel-Whittaker realization in a very simple natural way. 
We believe the similar argument is possible for the principal series: but in this case the natures of integrands of 
Whittaker case Ishii \cite{Ishii2}, Theorem 3.2  
and Siegel-Whittaker case \cite{Ishii}, Theorem 10.1 are still difficult to deform. 

For the $P_J$-principal series of the other group $\mathrm{SU(2,2)}$, we have explicit
integral expression of the matrix coefficients (Theorem 5.4 of \cite{KoOda}) similar to the above formula
(\ref{MC}).
We can expect the confluence from this formula to the integral expression of the Siegel-Whittaker
realization of \cite{Go} analogous to (\ref{SW}), corresponding to the limit of one parameter conjugations of 
a compact subgroup $K$ of $\mathrm{SU(2,2)}$:
\[ \lim_{t \to 0} h_t K h_t^{-1} = R. \] 
The deformation of the Siegel-Whittaker function of the $P_J$-principal series to the Whittaker function could
handled in a similar way as in \cite{HIO}.

In view of this observation, we may hope that similar phenomenon occurs for more general groups 
$\mathrm{SO}(2,q)$ $(q > 4)$.  

\section{Preliminaries}
\subsection{Basic notations}
Let $G$ be the real symplectic group of degree two:
\[ \mathrm{Sp}(2,\R) = \biggl\{  g \in \mathrm{SL}(4,\R) \, \bigg| \, 
{}^t g J_2 g = J_2 
= \begin{pmatrix}
0_2 & 1_2 \\
-1_2 & 0_2 
\end{pmatrix} 
\biggr\}, \]
with $1_2$ the unit matrix of degree two and 
$0_2$ the zero matrix of degree two. 

Fix a maximal compact subgroup $K$ of $G$ by
\[ K= \biggl\{  k(A,B) =  \begin{pmatrix}
A & B \\
-B & A 
\end{pmatrix}   \in G \, \bigg| \, 
A, B \in \mathrm{M}(2,\R) \biggr\}. \]
It is isomorphic to the unitary group $\mathrm{U}(2)$ via the homomorphism

\[ K \ni k(A,B) \mapsto A+\sqrt{-1}B \in \mathrm{U}(2). \]

We define a certain spherical subgroup $R$ of $G$ as follows.
Let $P_{S}= L_S \ltimes N_S$ be the Siegel parabolic subgroup with the Levi part $L_S$
and the abelian unipotent radical $N_S$ given by
\[ L_S = \biggl\{  \begin{pmatrix}
A & 0_2 \\
0_2 & {}^t A^{-1} 
\end{pmatrix}  \, \bigg| \, 
A \in \mathrm{GL}(2,\R) \biggr\}, \]
\[ N_S = \biggl\{  n(T) = \begin{pmatrix}
1_2 & T \\
0_2 & 1_2 
\end{pmatrix}  \, \bigg| \, 
{}^t T = T \in \mathrm{M}(2,\R) \biggr\}. \]
Fix a non-degenerate unitary character $\xi$ of $N_S$ by
\[  \xi(n(T)) = \exp \bigl( 2 \pi \sqrt{-1} \tr (H_{\xi} T) \bigr) \]
with $H_{\xi} = \bigl(  
\begin{smallmatrix}
h_1 & h_3/2 \\
h_3/2 & h_2
\end{smallmatrix} \bigr)
\in \mathrm{M}(2,\R)$ and $\det H_{\xi} \ne 0$. 
Consider the action of $L_S$ on $N_S$ by conjugation and the induced action on the character 
group $\widehat{N_S}$.
Define $\mathrm{SO}(\xi)$ to the identity component of the subgroup of $L_S$ which stabilize $\xi$:
\[  \mathrm{SO}(\xi) := \mathrm{Stab}_{L_{S}}(\xi)^{\circ}
  =  \biggl\{  \begin{pmatrix}
A & 0_2 \\
0_2 & {}^t A^{-1} 
\end{pmatrix}  \in L_{S} \, \bigg| \, 
{}^t A H_{\xi} A = H_{\xi} \biggr\}. \]
Then $\mathrm{SO}(\xi)$ is isomorphic to $\mathrm{SO}(2)$ if $\det H_{\xi}>0$ and
to $\mathrm{SO}_{\circ}(1,1)$ if $\det H_{\xi}<0$.
In this article we treat the case that $\xi$ is a `definite' character, that is $\det H_{\xi}>0$.
So we may assume $h_1,h_2>0$ and $h_3=0$ without loss of generality.
We sometimes identify the element of $\mathrm{SO}(\xi)$ with its upper left $2 \times 2$
component.
Fix a unitary character $\chi_{m_0}$ $(m_0 \in \Z)$ 
of $\mathrm{SO}(\xi) \cong \mathrm{SO}(2)$ by
\begin{equation} \label{chi}
\chi_{m_0} 
\biggl( 
\begin{pmatrix}
\sqrt{h_1}  &  \\
  & \sqrt{h_2} 
\end{pmatrix} ^{-1}
\begin{pmatrix}
\cos \theta  & \sin \theta \\
 - \sin \theta & \cos \theta 
\end{pmatrix} 
\begin{pmatrix}
\sqrt{h_1}  &  \\
  & \sqrt{h_2} 
\end{pmatrix} 
\biggr)
= \exp(\sqrt{-1}m_0 \theta). 
\end{equation}
We define 
\begin{equation} \label{R}
R=\mathrm{SO}(\xi) \ltimes N_{S}  \quad \mbox{ and } \quad 
\eta=\chi_{m_0}  \boxtimes \xi.
\end{equation}
  
Taking a maximal split torus $A$ of $G$ by
\[ A=\{ a=(a_1,a_2) = \mathrm{diag}(a_1,a_2,a_1^{-1},a_2^{-1}) \mid a_1, \, a_2>0  \}, 
\]
we have the decomposition $G=RAK$. 

\subsection{Siegel-Whittaker functions}
We consider the space $C_{\eta}^{\infty}(R \backslash G)$ of complex valued $C^{\infty}$
functions $f$ on $G$ satisfying 
\[ f(rg) = \eta(r) f(g) \quad \forall (r,g) \in R \times G. \]
By the right translation, $C_{\eta}^{\infty}(R \backslash G)$ is a smooth $G$-module and we denote
the same symbol its underlying $(\mathfrak{g}_{\C}, K)$-module. 
For an irreducible admissible representation $(\pi, H_{\pi})$ of $G$ and the 
subspace $H_{\pi, K}$ of $K$-finite vectors, the intertwining space 
\[  \mathcal{I}_{\eta, \pi} = \mathrm{Hom}_{(\mathfrak{g}_{\C}, K)}(H_{\pi, K}, C_{\eta}^{\infty}(R \backslash G))  \]
between the $(\mathfrak{g}_{\C}, K)$-modules is called the space of 
{\it algebraic Siegel-Whittaker functionals}. 
For a finite-dimensional $K$-module $(\tau, V_{\tau})$, denote by 
$C_{\eta, \tau}^{\infty}(R \backslash G/K)$ the space
\[ \bigl\{ \phi \colon G \to V_{\tau}, \, C^{\infty}
\mid \phi(r g k) = \eta(r) \tau(k^{-1}) \phi(g)
\quad \forall (r,g,k) \in R \times G \times K  \bigr\}.  \]
Let $(\tau^{*}, V_{\tau^{*}})$ be a $K$-type of $\pi$ and $\iota \colon V_{\tau^{*}} \to H_{\pi}$
be an injection. Here, $\tau^{*}$ means the contragredient representation of $\tau$. 
Then for $\Phi \in \mathcal{I}_{\eta, \pi}$, we can find an element $\phi_{\iota}$ in
\[ C_{\eta, \tau}^{\infty}(R \backslash G/K) 
= C_{\eta}^{\infty}(R \backslash G) \otimes V_{\tau^{*}} 
\cong \mathrm{Hom}_{K}(V_{\tau^{*}}, C_{\eta}^{\infty}(R \backslash G)) \]
via $\Phi(\iota(v^{*}))(g)=\langle v^{*}, \phi_{\iota}(g) \rangle$ with 
$\langle, \rangle$ the canonical pairing on $V_{\tau^{*}} \times V_{\tau}$. 

Since there is the decomposition $G=RAK$, our generalized spherical function $\phi_{\iota}$
is determined by its restriction $\phi_{\iota}|_{A}$, which we call the {\it radial part} of $\phi_{\iota}$.
For a subspace $X$ of  $C_{\eta, \tau}^{\infty}(R \backslash G/K)$, we denote
$X|_{A} =\{  \phi|_{A}  \in C^{\infty}(A) \mid \phi \in X \}$.

Let us define the space $\mathrm{SW}(\pi, \eta, \tau)$ of Siegel-Whittaker functions and its subspace 
$\mathrm{SW}(\pi, \eta, \tau)^{\text{rap}}$ as follows: 
\[  \mathrm{SW}(\pi, \eta, \tau) = \bigcup_{ \iota \in \mathrm{Hom}_K(\tau^{*}, \pi)}    
\{ \phi_{\iota} \mid \Phi \in \mathcal{I}_{\eta, \pi} \} \]
and 
\[  \mathrm{SW}(\pi, \eta, \tau)^{\text{rap}} 
= \bigl\{ \phi_{\iota} \in \mathrm{SW}(\pi, \eta, \tau) \mid 
\phi_{\iota} |_{A} \mbox{ decays rapidly as } a_1, a_2 \to \infty  \bigr\}. \]
We call an element in $\mathrm{SW}(\pi, \eta, \tau)$ a {\it Siegel-Whittaker function}
for $(\pi, \eta, \tau)$. 

\subsection{Parametrization of the discrete series representations}
Let $E_{ij} \in \mathrm{M}_4(\R)$ be the matrix unit with $1$ as its $(i,j)$-component and 
$0$ at the other entries.
The root system of $G$ with respect to a compact Cartan subgroup
\[ T=\exp \bigl( \R(E_{13}-E_{31})+ \R(E_{24}-E_{42}) \bigr) \]
is given by a set of vectors in the Euclidean plane:
\[ \{ \pm 2 \ve_1, \pm 2 \ve_2, \pm \ve_1 \pm \ve_2  \}. \]
Here, 
\begin{align*}
\ve_1 \bigl( r_1(E_{13}-E_{31})+r_2(E_{24}-E_{42}) \bigr) =& \sqrt{-1} r_1, \\
\ve_2 \bigl( r_1(E_{13}-E_{31})+r_2(E_{24}-E_{42}) \bigr) =& \sqrt{-1} r_2.
\end{align*}
We fix a subset of simple roots and the associated positive roots by
\[ \{ \ve_1-\ve_2, \ve_2 \}, \quad  \{ 2 \ve_1, \ve_1+\ve_2, \ve_1-\ve_2, 2 \ve_2   \}  \]
respectively.

Then the set of the unitary characters of $T$ (or their derivatives)
is identified naturally with $\Z \oplus \Z$, and the subset consisting of 
dominant integral weight is
\[ \Xi = \{ (n_1, n_2)  \in \Z \oplus \Z \mid n_1 \ge n_2 \}. \]
There is a bijection between $\widehat{K}$ and $\Xi$ 
by highest weight theory. Because the half-sum of the positive root is integral,
the discrete series representations of $G=\mathrm{Sp}(2,\R)$ are parametrized
by the subset of regular elements in $\Xi$:
\[ \Xi' = \{ (n_1,n_2) \in \Z \oplus \Z 
\mid n_1 > n_2, \, n_1 \ne 0, \, n_2 \ne 0, \, n_1+n_2 \ne 0  \}.   \]
Here the condition $n_1 > n_2$ means the positivity of weight $(n_1, n_2)$ with respect to the compact root
$\ve_1 - \ve_2=(1,-1)$.

The subsets $\Xi_{\text{I}}=\{(n_1,n_2) \mid n_1>n_2>0  \}$
and  $\Xi_{\text{IV}}=\{(n_1,n_2) \mid 0> n_1>n_2 \}$
parametrize  the holomorphic discrete series and the anti-holomorphic discrete 
series representations, respectively.
Set
\[ \Xi_{\text{II}}=\{(n_1,n_2) \mid n_1> 0> n_2, \, n_1+n_2 >0 \},  \]
and
\[ \Xi_{\text{III}}=\{(n_1,n_2) \mid n_1> 0> n_2, \, 0 > n_1+n_2 \}.  \]
Then the union $\Xi_{\text{II}} \cup \Xi_{\text{III}}$ parametrize the 
large discrete series representations of $G$.

\section{Miyazaki's results} 

Miyazaki derived a system of partial differential equations satisfied by Siegel-Whittaker
functions for the large discrete representations in \cite{Mi}. 
He also obtained multiplicity one property. We recall his results in this section.

Let $\tau=\tau_{(\lambda_1,\lambda_2)} = \mathrm{Sym}^{\lambda_1-\lambda_2}
\otimes \det^{\lambda_2}$ be the irreducible $K$-module with the highest weight $(\lambda_1, \lambda_2)$, 
then the dimension of $\tau$ is $d+1$ with $d=\lambda_1-\lambda_2$. 
We take the basis $\{ v_{j} \}_{j=0}^{d}$ of $V_{\tau^{*}}$ 
with $\tau^{*} = \tau_{(-\lambda_2,-\lambda_1)}$ as in
\cite[Lemma 3.1]{Mi}.

We remark on a compatibility condition. For a non-zero function $\phi$ in 
$C^{\infty}_{\eta,\tau^{*}}(R \backslash G /K)$, we have
\[ \phi(a) = \phi(mam^{-1}) = (\chi_{m_0} \boxtimes \xi)(m) \tau_{(-\lambda_2,-\lambda_1)}(m) 
\, \phi(a),
\]
where, $a \in A$ and $m \in \mathrm{SO}(\xi) \cap Z_{K}(A)=\{ \pm 1_{4} \}$. 
If we take $m=-1_{4}$, $(\chi_{m_0} \boxtimes \xi)(m)=\chi_{m_0}(m)=(-1)^{m_0}$ and
$\tau_{(-\lambda_2,-\lambda_1)}(m)=(-1)^{d}$ imply that $(m_0 +d)/2$ is an integer.  

\begin{proposition}[Miyazaki \cite{Mi}] \label{Miyazaki}
Let $\pi = \pi_{\Lambda}$ be a large discrete series representation of $G$ with 
the Harish-Chandra parameter
$\Lambda = (\lambda_1-1,\lambda_2) \in \Xi_{\mathrm{II}}$ and its minimal
$K$-type $\tau=\tau_{(\lambda_1,\lambda_2)}$.  
Let $\xi$ be a unitary character of $N_{S}$ associated with 
a positive definite matrix $H_{\xi} = \bigl(  
\begin{smallmatrix}
h_1 & 0 \\
0 & h_2
\end{smallmatrix} \bigr)$. Put $\eta=\chi_{m_0}  \boxtimes \xi $, as in $($\ref{chi}$)$. 
Then we have the following:
\begin{enumerate}
\item we have $\dim_{\C} \mathrm{SW}(\pi, \eta, \tau) \le 4$ and a function
\[  \phi_{SW}(a) = \sum_{j=0}^{d} \Bigl\{ 
(\sqrt{h_1}a_1)^{\lambda_1-j} (\sqrt{h_2}a_2)^{\lambda_2+j} e^{-2 \pi(h_1a_1^2+h_2a_2^2)}
c_j(a) \Bigr\} \, v_{j} \]
is in the space $\mathrm{SW}(\pi_{\Lambda}, \eta, \tau)|_{A}$ if and only if
$\{c_j(a)\}_{j=0}^{d}$ is a smooth solution of the following system:
\begin{equation}
\Bigl[ \partial_1+j\frac{h_2a_2^2}{\triangle} \Bigr] c_{j-1}(a) 
+ \sqrt{-1}m_0 \frac{h_2a_2^2}{\triangle} \, c_j(a)
-(d-j) \frac{h_2a_2^2}{\triangle} \, c_{j+1}(a)=0
\quad (1 \le j \le d),  \label{M1}
\end{equation}
\begin{equation}
j \frac{h_1a_1^2}{\triangle} \, c_{j-1}(a)
+ \sqrt{-1}m_0 \frac{h_1a_1^2}{\triangle} \, c_j(a)
+\Bigl[ \partial_2-(d-j) \frac{h_1a_1^2}{\triangle} \Bigr] c_{j+1}(a) 
=0
\quad (0 \le j \le d-1), \label{M2}
\end{equation}
\begin{align}
& h_1a_1^2 \Bigl[ \partial_2  - 8 \pi h_2 a_2^2 -2j \frac{h_2 a_2^2}{\triangle} +2 \lambda_2-2 \Bigr] c_{j-1}(a) 
- 2 \sqrt{-1} m_0 \frac{h_1a_1^2h_2a_2^2}{\triangle} \, c_j(a) \label{M3}
\\
& + h_2a_2^2 \Bigl[ \partial_1  - 8 \pi h_1 a_1^2 +2(d-j) \frac{h_1 a_1^2}{\triangle} +2 \lambda_2-2 \Bigr] c_{j+1}(a) 
=0 \quad (1 \le j \le d-1), \nonumber
\end{align}
with $\partial_i = a_i (\partial/\partial a_i)$ $(i=1,2)$ and $\triangle=h_1a_1^2 - h_2a_2^2$. 
\item  $\dim_{\C} \mathrm{SW}(\pi, \eta, \tau)^{\text{rap}} \le 1$. 
\end{enumerate}
\end{proposition}
This is a paraphrase of Propositions 10.2, 10.7 and Theorem 11.5 of \cite{Mi}.
Here (\ref{M1}), (\ref{M2}), (\ref{M3}) are essentially identical equations
to (10.4), (10.5), (10.6) of \cite{Mi} deduced from Proposition 10.2.
However we replace the symbol $\chi(Y_{\eta})$ of \cite{Mi} by
its explicit value $\sqrt{-1}m_0/\sqrt{h_1h_2}$, and the symbol $D$ of \cite{Mi} by $\triangle$.

\section{Partially Confluent hypergeometric functions in two variables} 

We introduce and study certain partially confluent hypergeometric functions in two variables  
on $A \simeq (\R_{>0})^2$. 
These functions play a key role in describing Siegel-Whittaker functions 
$\phi_{SW}(a)$. We remark that these types of  confluent hypergeometric functions have 
also appeared in \cite{Go}, for the large discrete series representations and $P_{J}$-principal series 
representations of $\mathrm{SU}(2,2)$,  
and \cite{HIO}, for the $P_{J}$-principal series representations of $\mathrm{Sp(2,\R)}$.  

\begin{definition}[Partially confluent hypergeometric functions]
\label{def:fk}
Let $\pi = \pi_{\Lambda}$ be a large discrete series representation of $G$ with 
the Harish-Chandra parameter
$\Lambda = (\lambda_1-1,\lambda_2) \in \Xi_{\mathrm{II}}$ and its minimal
$K$-type $\tau=\tau_{(\lambda_1,\lambda_2)}$.  
Let $\xi$ be a unitary character of $N_{S}$ associated with 
a positive definite matrix $H_{\xi} = \bigl(  
\begin{smallmatrix}
h_1 & 0 \\
0 & h_2
\end{smallmatrix} \bigr)$. 
Put $\eta=\chi_{m_0}  \boxtimes \xi $, as in $($\ref{chi}$)$. 

We assume that 
\[ |m_0| \ge d   \quad   \mbox{  and  } \quad  |m_0|\equiv d \pmod{2}.
\]
For $0 \le k \le d$, define 
\begin{align}
f_k(a) & = f_k(a; [\pi, \eta, \tau]) \nonumber \\
& = (\sqrt{h_1}a_1)^{2k+1} (\sqrt{h_2}a_2)^{2d+1-2k} (h_1a_1^2-h_2a_2^2)^{\frac{|m_0|-d}{2}} \\
& \times \int_{0}^{1}F \bigl( 2 \pi h_1 a_1^2 t +2 \pi h_2 a_2^2(1-t) \bigr) \, t^{\frac{|m_0|-d-1}{2}+k}
(1-t)^{\frac{|m_0|+d-1}{2}-k} \, dt, \nonumber
\end{align}
with 
\[ F(x) = e^x x^{-\frac{|m_0|+\lambda_1+1}{2}}
W_{\frac{\lambda_1-|m_0|-1}{2},\frac{\lambda_2}{2}}(2x).
\] 
Here, $W_{\kappa,\mu}(z)$ is Whittaker's confluent hypergeometric function. 
(See \cite{WW} Chapter 16 for definition.) 
\end{definition}
Since the indices $\alpha_k-1=\frac{|m_0|-d-1}{2}+k$, 
$\beta_k-1=\frac{|m_0|+d-1}{2}-k$ in the integrand of $f_k$
satisfy $\alpha_k,\beta_k>0$ $(0 \le k \le d)$, 
we see that
$f_k(a)$ is a smooth function on $A$ and of moderate growth when each $a_1,a_2$ tends to infinity.
We have further more,  

\begin{proposition} \label{prop:fk}
Partially confluent hypergeometric functions
$\{f_k(a)\}_{k=0}^{d}$ 
satisfy the following system of the difference-differential
equations.
\begin{equation} 
\partial _1 \, f_k = -(2k+1) \frac{h_2a_2^2}{\triangle} \, f_k 
+(|m_0|+d-1-2k) \frac{h_2a_2^2}{\triangle} \, f_{k+1} \quad (0 \le k \le d-1), 
\label{d_1fk}
\end{equation}

\begin{equation} 
\partial _2 \, f_k = (2d-2k+1) \frac{h_1a_1^2}{\triangle} \, f_k 
-(|m_0|-d-1+2k) \frac{h_1a_1^2}{\triangle} \, f_{k-1}  \quad (1 \le k \le d), 
\label{d_2fk}
\end{equation}

\begin{align} 
& \bigl[ (\partial_1+\partial_2)^2  +2(\lambda_2-2)(\partial_1+\partial_2)
-8 \pi (h_1 a_1^2 \partial_1+h_2 a_2^2 \partial_2)
-4(\lambda_2-1) 
 \bigr]  f_k (a) = 0 \label{omegafk} \\
 & (0 \le k \le d). \nonumber
\end{align} 
Here, $\triangle=h_1a_1^2-h_2a_2^2$. 
\end{proposition}
\begin{proof}
For $0 \le k \le d$, put 
\[ \check{f_k}(a)
= \int_{0}^{1}F \bigl( 2 \pi h_1 a_1^2 t +2 \pi h_2 a_2^2(1-t) \bigr) \, t^{\frac{|m_0|-d-1}{2}+k}
(1-t)^{\frac{|m_0|+d-1}{2}-k} \, dt. 
\]
Then we can verify that
\[
\partial_1 \, f_k =  (\sqrt{h_1}a_1)^{2k+1} (\sqrt{h_2}a_2)^{2d+1-2k} \triangle^{ \frac{|m_0|-d}{2}}
\Bigl[ \partial_1+(2k+1)+(|m_0|-d) \frac{h_1a_1^2}{\triangle} \Bigr] \check{f_k} \\ 
\]
and
\[
\partial_1 \, \check{f_k} =
-(|m_0|-d+1+2k) \frac{h_1a_1^2}{\triangle} \check{f_k}
+(|m_0|+d-1-2k) \frac{h_1a_1^2}{\triangle} \check{f_{k+1}}
\]
for $0 \le k \le d-1$. Therefore, we obtain the formula (\ref{d_1fk}).
Similarly, we have
\[
\partial_2 \, f_k =  (\sqrt{h_1}a_1)^{2k+1} (\sqrt{h_2}a_2)^{2d+1-2k} \triangle^{ \frac{|m_0|-d}{2}}
\Bigl[ \partial_2+(2d+1-2k)-(|m_0|-d) \frac{h_2a_2^2}{\triangle} \Bigr] \check{f_k} \\ 
\]
and
\[
\partial_2 \, \check{f_k} =
-(|m_0|-d-1+2k) \frac{h_2a_2^2}{\triangle} \check{f_{k-1}}
+(|m_0|+d+1-2k) \frac{h_2a_2^2}{\triangle} \check{f_{k}}
\]
for $1 \le k \le d$.
Thus, we have the formula (\ref{d_2fk}).

Let us  prove (\ref{omegafk}).
Put $\Omega$ and $\check{\Omega}$ be the partial differential operators defined by
\[ \Omega = (\partial_1+\partial_2)^2  +2(\lambda_2-2)(\partial_1+\partial_2)
-8 \pi (h_1 a_1^2 \partial_1+h_2 a_2^2 \partial_2)
-4(\lambda_2-1) \]
and 
\begin{align*}
\check{\Omega} =& (\partial_1+\partial_2)^2  +2(|m_0|+d+\lambda_2)(\partial_1+\partial_2)
-8 \pi (h_1 a_1^2 \partial_1+h_2 a_2^2 \partial_2) \\ 
& -8 \pi (h_1 a_1^2 +h_2 a_2^2 )(|m_0|+1) 
   -8 \pi (h_1 a_1^2 -h_2 a_2^2 )(2k-d)
   +(|m_0|+d)(|m_0|+d+2 \lambda_2).
\end{align*}
Then we can check that
\[ \Omega \, f_k =  (\sqrt{h_1}a_1)^{2k+1} (\sqrt{h_2}a_2)^{2d+1-2k} \triangle^{\frac{|m_0|-d}{2}} \, \check{\Omega} \, \check{f_k}.
\]
By interchanging differentiation and integration, we have
\[ \check{\Omega} \, \check{f_k}(a)
= \int_{0}^{1} G \bigl( 2 \pi h_1 a_1^2 t +2 \pi h_2 a_2^2(1-t) \bigr) \, t^{\frac{|m_0|-d-1}{2}+k}
(1-t)^{\frac{|m_0|+d-1}{2}-k} \, dt
\]
with
\[
G(x) = 4 \biggl[ 
x^2 \frac{d^2}{dx^2} -\bigl\{2x-(|m_0|+\lambda_1+1) \bigr\} x \frac{d}{dx} -2(|m_0|+1)x  
+ \frac{(|m_0|+\lambda_1)^2 - \lambda_2^2 }{4} \biggr] F(x). 
\]
Put $F(x)=e^x x^{-(|m_0|+\lambda_1+1)/2} H(x)$. Then we have
\[ 
G(x) = 4 e^x x^{-(|m_0|+\lambda_1+1)/2} \cdot x^2
\biggl[ 
\frac{d^2}{dx^2} + \Bigl\{ -1 + \frac{\lambda_1-|m_0|-1}{x}
+\frac{1-\lambda_2^2}{4x^2} \Bigr\}   
\biggr] H(x).       
\]    
It is known that the differential equation
\[ \biggl[ 
\frac{d^2}{dx^2} + \Bigl\{ -1 + \frac{\lambda_1-|m_0|-1}{x}
+\frac{1-\lambda_2^2}{4x^2} \Bigr\}   
\biggr] H(x) = 0 \]
has two linearly independent solutions:
\[
W_{\frac{\lambda_1-|m_0|-1}{2},\frac{\lambda_2}{2}}(2x), \quad
M_{\frac{\lambda_1-|m_0|-1}{2},\frac{\lambda_2}{2}}(2x).
\] 
Here, $W_{\kappa,\mu}(z)$ and $M_{\kappa,\mu}(z)$
are Whittaker's confluent hypergeometric functions.
(See \cite{WW} Chapter 16 for definition.) 
Therefore, we have (\ref{omegafk}).
It completes the proof.
\end{proof}

\section{Main results} 

We state our main results on an explicit integral formula of
Siegel-Whittaker functions which are of rapid decay 
for the large discrete series representations
of ${\rm Sp}(2,{\R})$. 

\begin{theorem} \label{th1}
Let $\pi = \pi_{\Lambda}$ be a large discrete series representation of $G$ with 
the Harish-Chandra parameter
$\Lambda = (\lambda_1-1,\lambda_2) \in \Xi_{\mathrm{II}}$ and its minimal
$K$-type $\tau=\tau_{(\lambda_1,\lambda_2)}$.  
Let $\xi$ be a unitary character of $N_{S}$ associated with 
a positive definite matrix $H_{\xi} = \bigl(  
\begin{smallmatrix}
h_1 & 0 \\
0 & h_2
\end{smallmatrix} \bigr)$. Put $\eta=\chi_{m_0}  \boxtimes \xi $, as in $($\ref{chi}$)$. 

We assume that 
\[ |m_0| \ge d   \quad   \mbox{  and  } \quad  |m_0|\equiv d \pmod{2}.
\]
Then we have the following:
\begin{enumerate}
\item $\dim_{\C} \mathrm{SW}(\pi, \eta, \tau)^{\text{rap}} = 1$.
\item 
Let $\{ f_k(a) \}_{k=0}^{d}$ be the partially confluent hypergeometric functions,  
defined in Definition \ref{def:fk}.
We consider the following $\C$-linear combinations $\{ g_j(a) \}_{j=0}^{d}$
of elements of $\{ f_k(a) \}_{k=0}^{d}$,  
given by
\begin{equation}
g_j(a) = \sum_{k=0}^{d} x_{jk} \, f_k(a), 
\end{equation} 
where the complex numbers  $\{ x_{jk} \}_{0 \le j,k \le d}$ are given by 
\begin{align} \label{def:x_jk}
x_{jk} =&  (-1)^{j+k} \bigl( \sgn(m_0) \sqrt{-1} \bigr)^{j} \binom{d}{2k-j} \, |m_0|^{\delta(j)}
    \prod_{l=1}^{k} \frac{2l-1}{|m_0|-d+2l-1} \\
          & \times \sum_{r=0}^{[j/2]}           
\binom{[j/2]}{r} \frac{\bigl( d-2r-\delta(j) \bigr)!}{d!}
\prod_{l=1}^{r} \biggl\{ \frac{2k-j-2l+2-\delta(j)}{2k-2l+1} \bigl( |m_0|^2-(d-2l+2)^2 \bigr) \biggr\}
\nonumber
\end{align}
$($$\delta(j) = 1$ if $j$ is odd, otherwise $0$$)$.
Then the function 
\begin{equation}
\phi_{SW}(a) = \sum_{j=0}^{d} \Bigl\{ 
(\sqrt{h_1}a_1)^{\lambda_1-j} (\sqrt{h_2}a_2)^{\lambda_2+j} e^{-2 \pi(h_1a_1^2+h_2a_2^2)}
g_j(a) \Bigr\} \, v_{j} 
\end{equation}
gives a non-zero element in $\mathrm{SW}(\pi, \eta, \tau)^{\text{rap}}|_{A}$
which is unique up to constant multiple.
\end{enumerate}
\end{theorem}

\section{Proof of main results} 

We claim that the coefficient functions $c_0(a),c_1(a),\dots,c_d(a)$ appearing in 
the Siegel-Whittaker function
$\phi_{SW}(a)$, in Proposition \ref{Miyazaki},  
are $\C$-linear combination of the confluent hypergeometric functions 
$f_0(a), f_1(a),\dots, f_d(a)$ defined in Definition \ref{def:fk}.

\begin{proposition} \label{prop:gj}
We assume that 
\[ |m_0| \ge d   \quad   \mbox{  and  } \quad  |m_0|\equiv d \pmod{2}.
\]
Let $\{ x_{jk} \}_{0 \le j,k \le d}$ be a sequence of complex numbers, which satisfy
\begin{equation} \label{d-1}
j \, x_{j-1,k} + \sqrt{-1} m_0 \, x_{j,k} +(d-2k+j+1) \, x_{j+1,k}
-(|m_0|-d+2k+1) \, x_{j+1,k+1} = 0
\end{equation}
$(0 \le j \le d-1, \, 0 \le k \le d)$, 
\begin{equation} \label{d-2}
(j-2k-1) \, x_{j-1,k} +(|m_0|+d-2k+1) \, x_{j-1,k-1}
+ \sqrt{-1} m_0 \, x_{j,k} - (d-j) \, x_{j+1,k}  = 0
\end{equation} 
$(1 \le j \le d, \, 0 \le k \le d)$, 
\begin{equation} \label{d-3}
x_{1,0}=x_{2,0}=\dots =x_{d,0}=0, 
\end{equation}
and
\begin{equation} \label{d-4}
x_{0,d}=x_{1,d}=\dots =x_{d-1,d}=0. 
\end{equation}
For $0 \le j \le d$, define
\[ g_j(a) = \sum_{k=0}^{d} x_{jk} \, f_k(a), \] 
then $\{ g_j(a) \}_{j=0}^{d}$ is a smooth solution of the system 
of the difference-differential equations
(\ref{M1}), (\ref{M2}) and (\ref{M3}) in Proposition \ref{Miyazaki}.
\end{proposition}
\begin{proof}
By using Proposition \ref{prop:fk}, we see that
\begin{align*}
&  \frac{\triangle}{h_2a_2^2} \Bigl[ \partial_1 + j \frac{h_2a_2^2}{\triangle} \Bigr] g_{j-1}(a) \\
& = \sum_{k=0}^{d} (j-2k-1) x_{j-1,k} \, f_k
    +\sum_{k=0}^{d} (|m_0|+d-2k-1) x_{j-1,k} \, f_{k+1}.
\end{align*}
Therefore, we have
\begin{align*}
&  \frac{\triangle}{h_2a_2^2} \Bigl[ \partial_1 +j  \frac{h_2a_2^2}{\triangle} \Bigr] g_{j-1}(a)  + \sqrt{-1} m_0 g_j(a)  -(d-j) g_{j+1}(a)   \\
& =  \sum_{k=0}^{d} (j-2k-1) x_{j-1,k} \, f_k + \sum_{k=1}^{d} (|m_0|+d-2k+1) x_{j-1,k-1} \, f_{k} \\
& \quad  + \sqrt{-1} m_0 \sum_{k=0}^{d} x_{j,k} \, f_k -(d-j) \sum_{k=0}^{d} x_{j+1,k} \, f_k \\ 
& = 0 \quad (1 \le j \le d).
\end{align*}
This is the desired (\ref{M1}) for $\{ g_j(a)\}_{j=0}^{d}$. In the above calculation, we used the relations (\ref{d-2}) and (\ref{d-4})
on $\{ x_{jk} \}_{0 \le j, k \le d}$. 
Again, by using Proposition \ref{prop:fk}, we see that
\begin{align*}
&  \frac{\triangle}{h_1a_1^2} \Bigl[ \partial_2 -(d-j) \frac{h_1a_1^2}{\triangle} \Bigr] g_{j+1}(a) \\
& = \sum_{k=0}^{d} (d-2k+j+1) x_{j+1,k} \, f_k
    -\sum_{k=0}^{d} (|m_0|-d+2k-1) x_{j+1,k} \, f_{k-1}.
\end{align*}
Therefore, we have
\begin{align*}
&  j g_{j-1}(a)  + \sqrt{-1} m_0 g_j(a) +\frac{\triangle}{h_1a_1^2} \Bigl[ \partial_2 -(d-j) \frac{h_1a_1^2}{\triangle} \Bigr] g_{j+1}(a) \\
& =  j \sum_{k=0}^{d} x_{j-1,k} \, f_k + \sqrt{-1} m_0 \sum_{k=0}^{d} x_{j,k} \, f_k  +  \sum_{k=0}^{d} (d-2k+j+1) x_{j+1,k} \, f_k \\
&   \quad  -\sum_{k=0}^{d-1} (|m_0|-d+2k+1) x_{j+1,k+1} \, f_{k} \\
& = 0 \quad (0 \le j \le d-1). 
\end{align*}
This is the desired (\ref{M2}) for $\{ g_j(a)\}_{j=0}^{d}$. In the above calculation, we used the relations (\ref{d-1}) and (\ref{d-3})
on $\{ x_{jk} \}_{0 \le j, k \le d}$. 

By considering $(\ref{M1}) \times h_1a_1^2 + (\ref{M2}) \times h_2a_2^2 + (\ref{M3})$, we have
\begin{equation} \label{M4}
\begin{split} 
&h_1a_1^2[(\partial_1+\partial_2) -8 \pi h_2a_2^2 +2 \lambda_2-2 ] c_{j-1} (a) \\ 
&+ h_2a_2^2[(\partial_1+\partial_2) -8 \pi h_1a_1^2 +2 \lambda_2-2 ] c_{j+1} (a)=0  \quad (1 \le j \le d-1). 
\end{split}
\end{equation}
By considering $(\ref{M1}) \times h_1a_1^2 - (\ref{M2}) \times h_2a_2^2 $, we have
\begin{equation}
h_1a_1^2 \partial_1 \, c_{j-1}(a)  
-   h_2a_2^2 \partial_2 \, c_{j+1}(a) =0  \quad (1 \le j \le d-1). \label{M5} 
\end{equation}
Operating $\partial_2 (h_2a_2^2)^{-1}$ on both hand sides of $(\ref{M4})$, we have
\begin{equation} \label{M42}
\begin{split} 
& \partial_2  \frac{h_1a_1^2}{h_2a_2^2} [(\partial_1+\partial_2) -8 \pi h_2a_2^2 +2 \lambda_2-2 ] c_{j-1} (a) \\ 
&+ [(\partial_1+\partial_2) -8 \pi h_1a_1^2 +2 \lambda_2-2 ] \partial_2 \, c_{j+1} (a)=0  \quad (1 \le j \le d-1). 
\end{split}
\end{equation}
Combining (\ref{M42}) and (\ref{M5}), we have, for $0 \le j \le d-2$,   
\begin{equation} \label{M6}
\bigl[ (\partial_1+\partial_2)^2  +2(\lambda_2-2)(\partial_1+\partial_2)\partial_2
-8 \pi (h_1 a_1^2 \partial_1+h_2 a_2^2 \partial_2)
-4(\lambda_2-1) \bigr] c_j(a) = 0. \\
\end{equation}
Operating $\partial_1 (h_1a_1^2)^{-1}$ on both hand sides of $(\ref{M4})$ and combining with (\ref{M5}), 
we have (\ref{M6}) for $2 \le j \le d$. As a result, we have (\ref{M6}) for $0 \le j \le d$. 

Lastly we prove that $\{g_j(a)\}_{j=0}^{d}$ satisfy $(\ref{M3})$.  
By Proposition \ref{prop:fk}, $f_0(a), f_1(a),\dots, f_d(a)$ satisfy the same differential equation (\ref{M6}),
therefore their $\C$-linear combinations $\{g_j(a)\}_{j=0}^{d}$ also satisfy (\ref{M6}).
Since $\{ g_j(a) \}_{j=0}^{d}$ satisfy (\ref{M5}) and (\ref{M6}), we see that
\begin{equation} \label{Fj}
\partial_2 \frac{F_j(a)}{h_2a_2^2} = \partial_1 \frac{F_j(a)}{h_1a_1^2} = 0 \quad (1 \le j \le d-1),
\end{equation}
where we put 
\begin{align*}
F_j(a)  =&  h_1a_1^2[(\partial_1+\partial_2) -8 \pi h_2a_2^2 +2 \lambda_2-2 ] g_{j-1} (a) \\
&+ h_2a_2^2[(\partial_1+\partial_2) -8 \pi h_1a_1^2 +2 \lambda_2-2 ] g_{j+1} (a). 
\end{align*}
By (\ref{Fj}), there exist constants $\beta_j $ $(1\le j \le d-1)$ such that
\[ F_j(a) = \beta_j \, h_1a_1^2 h_2a_2^2.
\]

Let us determine $\beta_j$. For $y>0$, define $L=\{ (a_1,a_2) \in A \mid h_1a_1^2= h_2a_2^2 = y  \}$.
We show that $(F_{j}|_{L})(y)=0$ to deduce $\beta_j = 0$. 
We can verify that
\begin{align*}
(\partial_1+\partial_2) \, f_k 
& = (\sqrt{h_1}a_1)^{2k+1} (\sqrt{h_2}a_2)^{2d+1-2k} \triangle^{\frac{|m_0|-d}{2}}
\bigl[ (\partial_1+\partial_2) +(d+|m_0|+2) \bigr]  \\
& \quad \times \int_{0}^{1}F \bigl( 2 \pi h_1 a_1^2 t +2 \pi h_2 a_2^2(1-t) \bigr) \, t^{\frac{|m_0|-d-1}{2}+k}
(1-t)^{\frac{|m_0|+d-1}{2}-k} \, dt \\
& = (\sqrt{h_1}a_1)^{2k+1} (\sqrt{h_2}a_2)^{2d+1-2k} \triangle^{\frac{|m_0|-d}{2}} \\
& \quad \times \int_{0}^{1}F^{*} \bigl( 2 \pi h_1 a_1^2 t +2 \pi h_2 a_2^2(1-t) \bigr) \, t^{\frac{|m_0|-d-1}{2}+k}
(1-t)^{\frac{|m_0|+d-1}{2}-k} \, dt.
\end{align*}
Here, 
\[ F^{*}(x) = \biggl( \Bigl( 
2x \frac{d}{dx} +(d+|m_0|+2) 
\Bigr) F \biggr)(x).
\]
There are two cases:
\begin{enumerate}
\item If $|m_0|-d \ge 2$, then both $(\partial_1+\partial_2)  f_k$ and $f_k$ have zeros at
$\{ h_1a_1^2-h_2a_2^2 = 0\}$. Hence, both $(\partial_1+\partial_2) g_j$ and $g_j$ have zeros at
there. Therefore we have $(F_{j}|_{L})(y)=0$. 
\item Suppose that $|m_0|=d$, then we have
\[  \bigl( (\partial_1+\partial_2) \, f_k \bigr) \big|_{L}(y)
= y^{d+1} F^{*}(2 \pi y) \, B \Bigl( k+\frac{1}{2}, d-k+\frac{1}{2} \Bigr).
\]
Here, 
\[ B(\alpha, \beta) = \int_{0}^{1} t^{\alpha-1}(1-t)^{\beta-1} \, dt
  = \frac{\Gamma(\alpha) \, \Gamma(\beta)}{\Gamma(\alpha+\beta)}
\]
is the Beta function.
Then we can verify that
\begin{equation} \label{eq:F*}
\begin{split}
 & \bigl( (\partial_1+\partial_2) \, g_j \bigr) \big|_{L}(y) 
 = \biggl( \sum_{k=0}^{d} x_{jk} \, B \Bigl( k+\frac{1}{2}, d-k+\frac{1}{2} \Bigr) \biggr) y^{d+1} F^{*}(2 \pi y)  \\
 & = \bigl( \sgn(m_0) \sqrt{-1} \bigr)^{j} \, 2^{-d} \pi \, 
 y^{d+1} F^{*}(2 \pi y).
 \end{split}  
\end{equation}
To derive (\ref{eq:F*}), we used (\ref{xjk}) in Proposition \ref{prop:x_jk}: (We will prove later.)
\[ x_{jk} = (-1)^{j+k} \bigl( \sgn(m_0) \sqrt{-1} \bigr)^{j} \binom{d}{2k-j}  \quad (\mbox{when } |m_0|=d), 
\]
under the condition $x_{0,0}=1$, and the equality:
\begin{equation} \label{eq:binom}
\sum_{k=0}^{d} (-1)^k \binom{2k}{j} \binom{2d-2k}{d-j} \binom{d}{k}
  =  (-1)^j 2^d \binom{d}{j}. 
\end{equation} 
Let $S_{d,j}$ be the left hand side of (\ref{eq:binom}). We remark that the above equality (\ref{eq:binom}) 
is proved by showing that
\[ S_{d,j} = -\frac{2d}{j} \, S_{d-1,j-1}  \quad (j \ge 1)
\mbox{ and } S_{d,0} =2^d. \]
See p.620 no. 63
in \cite{MF} for the equality $S_{d,0} =2^d$.

Similarly, we also obtain  
\[ g_j \big|_{L}(y) 
=  \bigl( \sgn(m_0) \sqrt{-1} \bigr)^{j} \, 2^{-d} \pi \, 
 y^{d+1} F(2 \pi y).   \]
Therefore, we have
\[  \bigl( (\partial_1+\partial_2) \, (g_{j-1} + g_{j+1}) \bigr) \big|_{L}(y)
   =   (g_{j-1}+g_{j+1} ) \big|_{L}(y) = 0,
\] 
for $1 \le j \le d-1$. Therefore we have $(F_{j}|_{L})(y)=0$. 
\end{enumerate}
In any cases, $\{ g_j(a) \}_{j=0}^{d}$ satisfy (\ref{M4}). 
Hence, $\{ g_j(a) \}_{j=0}^{d}$ satisfy (\ref{M3}) and it completes the proof.
\end{proof}

\begin{proposition} \label{prop:x_jk}
We assume that 
\[ |m_0| \ge d   \quad   \mbox{  and  } \quad  |m_0|\equiv d \pmod{2}.
\]
Let $\{ x_{jk} \}_{0 \le j,k \le d}$ be a sequence of complex numbers, which satisfy
\begin{equation} \label{x-1}
j \, x_{j-1,k} + \sqrt{-1} m_0 \, x_{j,k} +(d-2k+j+1) \, x_{j+1,k}
-(|m_0|-d+2k+1) \, x_{j+1,k+1} = 0
\end{equation}
$(0 \le j \le d-1, \, 0 \le k \le d)$, 
\begin{equation} \label{x-2}
(j-2k-1) \, x_{j-1,k} +(|m_0|+d-2k+1) \, x_{j-1,k-1}
+ \sqrt{-1} m_0 \, x_{j,k} - (d-j) \, x_{j+1,k}  = 0
\end{equation} 
$(1 \le j \le d, \, 0 \le k \le d)$, 
\begin{equation} \label{x-3}
x_{1,0}=x_{2,0}=\dots =x_{d,0}=0, 
\end{equation}
and
\begin{equation} \label{x-4}
x_{0,d}=x_{1,d}=\dots =x_{d-1,d}=0. 
\end{equation}
Then the sequence $\{ x_{jk} \}_{0 \le j,k \le d}$ is uniquely determined and 
given by, up to a constant multiple, 
\begin{equation} \label{xjk}
\begin{split}
x_{jk} =&  (-1)^{j+k} \bigl( \sgn(m_0) \sqrt{-1} \bigr)^{j} \binom{d}{2k-j} \, |m_0|^{\delta(j)}
    \prod_{l=1}^{k} \frac{2l-1}{|m_0|-d+2l-1} \\
          & \times \sum_{r=0}^{[j/2]}           
\binom{[j/2]}{r} \frac{\bigl( d-2r-\delta(j) \bigr)!}{d!}
\prod_{l=1}^{r} \biggl\{ \frac{2k-j-2l+2-\delta(j)}{2k-2l+1} \bigl( |m_0|^2-(d-2l+2)^2 \bigr) \biggr\}, 
\end{split}
\end{equation}
where, $\delta(j) = 1$ if $j$ is odd, otherwise $0$. 
\end{proposition}
\begin{proof}
We may assume that $x_{0,0}=1$. 
Let us write down (\ref{x-1}) with $j=0,1$ and $k$ replaced by $k-1,k$,
and (\ref{x-2}) with $j=1$ and $k$ replaced by $k-1,k,k+1$.
Then we have the following system of seven difference equations:
\begin{equation} \label{eq:seven}
\begin{split} 
& \sqrt{-1} m_0 \, x_{0,k-1} +(d-2k+3) \, x_{1,k-1}
-(|m_0|-d+2k-1) \, x_{1,k} = 0, \\
& \sqrt{-1} m_0 \, x_{0,k} +(d-2k+1) \, x_{1,k}
-(|m_0|-d+2k+1) \, x_{1,k+1} = 0, \\
& x_{0,k-1} + \sqrt{-1} m_0 \, x_{1,k-1} +(d-2k+4) \, x_{2,k-1}
-(|m_0|-d+2k-1) \, x_{2,k} = 0, \\
& x_{0,k} + \sqrt{-1} m_0 \, x_{1,k} +(d-2k+2) \, x_{2,k}
-(|m_0|-d+2k+1) \, x_{2,k+1} = 0, \\
&-(2k-2) \, x_{0,k-1} +(|m_0|+d-2k+3) \, x_{0,k-2}
+ \sqrt{-1} m_0 \, x_{1,k-1} - (d-1) \, x_{2,k-1}  = 0, \\
&-2k \, x_{0,k} +(|m_0|+d-2k+1) \, x_{0,k-1}
+ \sqrt{-1} m_0 \, x_{1,k} - (d-1) \, x_{2,k}  = 0, \\
&-(2k+2) \, x_{0,k+1} +(|m_0|+d-2k-1) \, x_{0,k}
+ \sqrt{-1} m_0 \, x_{1,k+1} - (d-1) \, x_{2,k+1}  = 0.
\end{split}
\end{equation} 
We eliminate $x_{j,k}, \, x_{j,k\pm1}$ $(j=1,2)$ in the above system. Then we have
\begin{equation} \label{eq:x0k}
\begin{split}
& -2(k+1)(|m_0|-d+2k+1)(|m_0|-d+2k-1) \, x_{0,k+1} \\
& +\bigl\{ 2k(4d-6k+3)-d(d-1) \bigr\} (|m_0|-d+2k-1) \, x_{0,k} \\
& +\Bigl[ 2(d-2k+2) \bigl\{ d(d-1) -(2k-2)(2d-2k+3) \bigr\} \\
& \quad  +(2k-2)(|m_0|+d-2k+1)(|m_0|-d+2k-1) \Bigr] \, x_{0,k-1} \\
& +(d-2k+4)(d-2k+3)(|m_0|+d-2k+3) \, x_{0,k-2} \\
& =0.
\end{split}
\end{equation} 
By solving the difference equation (\ref{eq:x0k}) for $\{x_{0,k}\}$ with $x_{0,0}=1$,
we obtain 
\begin{equation} \label{x0k}
x_{0,k} = (-1)^k \binom{d}{2k} 
    \prod_{l=1}^{k} \frac{2l-1}{|m_0|-d+2l-1}.
\end{equation} 

From the second equation of (\ref{eq:seven}), (\ref{x0k}) and (\ref{x-3}), 
we have the following difference equation for $\{x_{1,k}\}$: 
\[ (|m_0|-d+2k+1) \, x_{1,k+1} 
=  (d-2k+1) \, x_{1,k} + \sqrt{-1} m_0 \, (-1)^k \binom{d}{2k} 
    \prod_{l=1}^{k} \frac{2l-1}{|m_0|-d+2l-1}
\]
with $x_{1,0}=0$. Thus, we obtain 
\begin{equation} \label{x1k}
x_{1,k} = (-1)^{k+1} \bigl( \sgn(m_0) \sqrt{-1} \bigr) \binom{d}{2k-1} 
\frac{|m_0|}{d}
\prod_{l=1}^{k} \frac{2l-1}{|m_0|-d+2l-1}.
\end{equation} 

We prove the formula (\ref{xjk}) for general $x_{jk}$ by induction on the index $j$.
For $0 \le 2p, 2p+1,k \le d$, let us  prove
\begin{equation} \label{xjk2p}
\begin{split}
x_{2p, k} =&  (-1)^{k+p}  \binom{d}{2k-2p} 
    \prod_{l=1}^{k} \frac{2l-1}{|m_0|-d+2l-1} \\
          & \times \sum_{r=0}^{p}           
\binom{p}{r} \frac{\bigl( d-2r \bigr)!}{d!}
\prod_{l=1}^{r} \biggl\{ \frac{2k-2p-2l+2}{2k-2l+1} \bigl( |m_0|^2-(d-2l+2)^2 \bigr) \biggr\}, 
\end{split}
\end{equation}
\begin{equation} \label{xjk2p+1}
\begin{split}
x_{2p+1, k} =&  (-1)^{k+p+1} \bigl( \sgn(m_0) \sqrt{-1} \bigr) \binom{d}{2k-2p-1} \, |m_0|
    \prod_{l=1}^{k} \frac{2l-1}{|m_0|-d+2l-1} \\
          & \times \sum_{r=0}^{p}           
\binom{p}{r} \frac{\bigl( d-2r-1 \bigr)!}{d!}
\prod_{l=1}^{r} \biggl\{ \frac{2k-2p-2l}{2k-2l+1} \bigl( |m_0|^2-(d-2l+2)^2 \bigr) \biggr\}.
\end{split}
\end{equation}
Suppose that $2p+1 \le d$ and the above formulas 
are true for $x_{2p-1,k}, \, x_{2p,k}$ ($0 \le k \le d$), 
then substitute them into (\ref{x-2}):
\begin{equation} \label{deq:x2p+1k}
x_{2p+1,k}  = 
\frac{2p-2k-1}{d-2p} \, x_{2p-1,k} +\frac{|m_0|+d-2k+1}{d-2p} \, x_{2p-1,k-1}
+\frac{\sqrt{-1} \, m_0}{d-2p} \, x_{2p,k}.
\end{equation} 
Put 
\[ a(k) = \prod_{l=1}^{k} \frac{2l-1}{|m_0|-d+2l-1},  \quad
   b(r;k,p) = \prod_{l=1}^{r} \biggl\{ \frac{2k-2p-2l}{2k-2l+1} \bigl( |m_0|^2-(d-2l+2)^2 \bigr) \biggr\}.
\]
Then, we have
\begin{align} \label{2p-1,k}
& \frac{(-1)^{k+p+1}(2p-2k-1)}{\sqrt{-1} \, m_0 (d-2p)} \, x_{2p-1,k} \\
& = - \frac{2p-2k-1}{d-2p} 
\binom{d}{2k-2p+1} a(k)
\sum_{r=0}^{p-1}           
\binom{p-1}{r} \frac{\bigl( d-2r-1 \bigr)!}{d!} \, b(r;k,p-1) \nonumber \\
&= \frac{a(k)}{d-2p} \binom{d}{2k-2p-1} (d-2k+2p+1)(d-2k+2p) \nonumber \\   
 & \quad \times \sum_{r=0}^{p-1}        
\binom{p-1}{r} \frac{\bigl( d-2r-1 \bigr)!}{d!} \, \frac{b(r;k,p)}{2k-2p-2r},
\nonumber 
\end{align}
and
\begin{align} \label{2p-1,k-1}
& \frac{(-1)^{k+p+1}(|m_0|+d-2k+1)}{\sqrt{-1} \, m_0 (d-2p)} \, x_{2p-1,k-1} \\
& = \frac{|m_0|+d-2k+1}{d-2p} 
\binom{d}{2k-2p-1} a(k-1)
\sum_{r=0}^{p-1}           
\binom{p-1}{r} \frac{\bigl( d-2r-1 \bigr)!}{d!} \, b(r;k-1,p-1) \nonumber \\
&  = \frac{a(k)}{d-2p} \binom{d}{2k-2p-1}
\bigl\{ |m_0|^2-(d-2k+1)^2 \bigr\}
\sum_{r=0}^{p-1}           
\binom{p-1}{r} \frac{\bigl( d-2r-1 \bigr)!}{d!} \, \frac{b(r;k,p)}{2k-2r-1} \nonumber \\
& = \frac{a(k)}{d-2p} \binom{d}{2k-2p-1}
\sum_{r=0}^{p-1}           
\binom{p-1}{r} \frac{\bigl( d-2r-1 \bigr)!}{d!} \, \frac{b(r+1;k,p)}{2k-2p-2r-2} \nonumber \\
& \quad + 
\frac{a(k)}{d-2p} \binom{d}{2k-2p-1}
\sum_{r=0}^{p-1}           
\binom{p-1}{r} \frac{\bigl( d-2r-1 \bigr)!}{d!} \, (2d-2r-2k+1) \, b(r;k,p).
\nonumber 
\end{align}
Here, we used the equality: 
\[  \bigl\{ |m_0|^2-(d-2k+1)^2 \bigr\} =
     \bigl\{ |m_0|^2-(d-2r)^2 \bigr\} + (2d-2r-2k+1)(2k-2r-1)
\]    
in the above calculation. By noting the equality:
\[  \frac{(d-2k+2p+1)(d-2k+2p)}{2k-2p-2r} + (2d-2r-2k+1)
= \frac{(d-2r)(d-2r+1)}{2k-2p-2r}  -2p,
\]
we have
\begin{align*}
& \mbox{(\ref{2p-1,k})+(\ref{2p-1,k-1})} \\
& = \frac{a(k)}{d-2p} \binom{d}{2k-2p-1}
\sum_{r=0}^{p-1}           
\binom{p-1}{r} \frac{\bigl( d-2r-1 \bigr)!}{d!} 
\bigg\{  \frac{(d-2r)(d-2r+1)}{2k-2p-2r}  -2p \biggr\} 
b(r;k,p) \\
& \quad +  \frac{a(k)}{d-2p} \binom{d}{2k-2p-1}
\sum_{r=1}^{p}           
\binom{p-1}{r-1} \frac{\bigl( d-2r+1 \bigr)!}{d!} \, \frac{b(r;k,p)}{2k-2p-2r} \\
& = \frac{a(k)}{d-2p} \binom{d}{2k-2p-1}
\sum_{r=0}^{p}           
\binom{p}{r} \frac{\bigl( d-2r+1 \bigr)!}{d!} \, \frac{b(r;k,p)}{2k-2p-2r} \\
& \quad -2p \, 
\frac{a(k)}{d-2p} \binom{d}{2k-2p-1}
\sum_{r=0}^{p-1}           
\binom{p-1}{r} \frac{\bigl( d-2r-1 \bigr)!}{d!} \, b(r;k,p). 
\end{align*}

On the other hand, we have
\begin{align}
& \frac{(-1)^{k+p+1}}{\sqrt{-1} \, m_0 (d-2p)} \, x_{2p,k} \label{2p,k} \\
& = -\frac{1}{d-2p} 
\binom{d}{2k-2p} a(k)
\sum_{r=0}^{p}           
\binom{p}{r} \frac{\bigl( d-2r \bigr)!}{d!} \, b(r;k,p-1) \nonumber \\
& = -\frac{a(k)}{d-2p} 
\binom{d}{2k-2p-1} (d-2k+2p+1)
\sum_{r=0}^{p}           
\binom{p}{r} \frac{\bigl( d-2r \bigr)!}{d!} \, \frac{b(r;k,p)}{2k-2p-2r} 
\nonumber \\
& = \frac{a(k)}{d-2p} 
\binom{d}{2k-2p-1} 
\sum_{r=0}^{p}           
\binom{p}{r} \frac{\bigl( d-2r \bigr)!}{d!} \, b(r;k,p) \nonumber \\
& \quad - \frac{a(k)}{d-2p} 
\binom{d}{2k-2p-1} 
\sum_{r=0}^{p}           
\binom{p}{r} \frac{\bigl( d-2r +1 \bigr)!}{d!} \, \frac{b(r;k,p)}{2k-2p-2r}.
\nonumber 
\end{align}
Therefore, we have
\begin{align*}
& \mbox{(\ref{2p-1,k})+(\ref{2p-1,k-1})+(\ref{2p,k})} \\
& = \frac{a(k)}{d-2p} \binom{d}{2k-2p-1} 
\sum_{r=0}^{p}           
\binom{p}{r} \frac{\bigl( d-2r \bigr)!}{d!} \, b(r;k,p) \\
& \quad -2p \, 
\frac{a(k)}{d-2p} \binom{d}{2k-2p-1}
\sum_{r=0}^{p-1}           
\binom{p-1}{r} \frac{\bigl( d-2r-1 \bigr)!}{d!} \, b(r;k,p) \\
& = a(k) \, \binom{d}{2k-2p-1} 
\sum_{r=0}^{p}           
\binom{p}{r} \frac{\bigl( d-2r-1 \bigr)!}{d!} \, b(r;k,p). 
\end{align*}
By (\ref{deq:x2p+1k}) and the above formula, we obtain
\[ x_{2p+1,k} = (-1)^{k+p+1} \bigl( m_0 \sqrt{-1} \bigr) 
a(k) \, \binom{d}{2k-2p-1} 
\sum_{r=0}^{p}           
\binom{p}{r} \frac{\bigl( d-2r-1 \bigr)!}{d!} \, b(r;k,p).
\] 
Therefore, the formula is valid for $x_{2p+1,k}$ ($0 \le k \le d$).
Similarly, 
if we suppose that $2p \le d$ and the formulas 
(\ref{xjk2p}) and (\ref{xjk2p+1})
are true for $x_{2p-1,k}, \, x_{2p-2,k}$ ($0 \le k \le d$), 
then we can prove the formula is also valid for $x_{2p,k}$ ($0 \le k \le d$).
It completes the proof.
\end{proof}

Let us complete the proof of Theorem \ref{th1}. \\
{\it Proof of \ref{th1}}. By the condition
$|m_0| \ge d$, $|m_0|\equiv d \pmod{2}$, 
Propositions \ref{prop:gj} and \ref{prop:x_jk}, $\{ g_{j}(a) \}_{j=0}^{d}$
is a {\it non-zero} smooth solution of the system (\ref{M1}), (\ref{M2}) and (\ref{M3}) 
in Proposition \ref{Miyazaki}.
Furthermore, we can check that all of $\{ e^{-2 \pi(h_1a_1^2+h_2a_2^2)} \, g_{j}(a) \}_{j=0}^{d}$ 
are rapidly decreasing when each $a_1,a_2$ tends to infinity, by Definition \ref{def:fk}. 
We completes the proof. \hfill $\qed$

\section{Remarks on $\{ x_{jk} \}$} 

We have some remarks on the sequence of complex numbers $\{ x_{jk} \}_{0 \le j,k \le d}$
appearing in Theorem \ref{th1} and Proposition \ref{prop:x_jk}. 
Let $X_d = (x_{jk})_{0 \le j,k \le d} \in M_{d+1}(\C)$ be the square matrix of size $(d+1)$,
defined by the sequence $\{ x_{jk} \}_{0 \le j,k \le d}$.

Put $\varepsilon = \sgn(m_0) \sqrt{-1}$ and $t=|m_0|$.
Then $x_{jk}$ is given by
\begin{equation} 
\begin{split}
x_{jk} =&  (-1)^{j+k} \varepsilon^{j} \binom{d}{2k-j} \, t^{\delta(j)}
    \prod_{l=1}^{k} \frac{2l-1}{t-d+2l-1} \\
          & \times \sum_{r=0}^{[j/2]}           
\binom{[j/2]}{r} \frac{\bigl( d-2r-\delta(j) \bigr)!}{d!}
\prod_{l=1}^{r} \biggl\{ \frac{2k-j-2l+2-\delta(j)}{2k-2l+1} \bigl( t^2-(d-2l+2)^2 \bigr) \biggr\}, 
\end{split}
\end{equation}
where, $\delta(j) = 1$ if $j$ is odd, otherwise $0$. 

Furthermore, we set
\begin{equation}
Z_{d} = (z_{jk})_{0 \le j,k \le d} \in M_{d+1}(\C) \mbox{ with } 
z_{jk} = \varepsilon^{-j} x_{jk}. 
\end{equation}

Though the Harish-Chandra parameter $\Lambda = (\lambda_1-1,\lambda_2) \in \Xi_{\mathrm{II}}$
implies that $d=\lambda_1-\lambda_2 \ge 4$, 
we formally write down matrices $Z_d$ for $1 \le d \le 7$. 

\begin{example}[$Z_d$ for $1 \le d \le 7$]
\[
Z_1 = 
\begin{bmatrix}
1 & 0 \\
0 & 1 
\end{bmatrix}, \quad
Z_2 = 
\begin{bmatrix}
1 & - \frac{1}{t-1} & 0 \\
0 & \frac{t}{t-1} & 0 \\
0 & -\frac{1}{t-1} & 1
\end{bmatrix}, 
\quad 
Z_3 = 
\begin{bmatrix}
1 & -\frac{3}{t-2} & 0 & 0  \\
0 & \frac{t}{t-2} & - \frac{1}{t-2} & 0 \\
0 & -\frac{1}{t-2} & \frac{t}{t-2} & 0 \\
0 &  0 & -\frac{3}{t-2} & 1 
\end{bmatrix}, 
\]
\[
Z_4=
\begin{bmatrix}
1 & -\frac{6}{t-3} & \frac{3}{(t-3)(t-1)} & 0 & 0  \\
0 &  \frac{t}{t-3}   & - \frac{3 t}{(t-3)(t-1)} & 0  & 0 \\
0 & -\frac{1}{t-3} & \frac{t^2+2}{(t-3)(t-1)} & -\frac{1}{t-3} & 0 \\
0 &                     0 & -\frac{3 t}{(t-3)(t-1)} &  \frac{t}{t-3} & 0 \\
0 &                     0 & \frac{3}{(t-3)(t-1)} & -\frac{6}{t-3} &  1
\end{bmatrix},
\]
\[
Z_5=
\begin{bmatrix}
1 & -\frac{10}{t-4} & \frac{15}{(t-4)(t-2)} & 0 & 0 & 0 \\
0 & \frac{t}{t-4} & -\frac{6t}{(t-4)(t-2)} & \frac{3}{(t-4)(t-2)}  & 0 & 0 \\
0 & -\frac{1}{t-4} & \frac{t^2+5}{(t-4)(t-2)} & -\frac{3t}{(t-4)(t-2)} & 0 & 0\\
0 &  0 & -\frac{3 t}{(t-4)(t-2)} &  \frac{t^2+5}{(t-4)(t-2)} & - \frac{1}{t-4} & 0 \\
0 & 0 & \frac{3}{(t-4)(t-2)} & - \frac{6t}{(t-4)(t-2)} &  \frac{t}{t-4} & 0 \\
0 & 0 & 0 & \frac{15}{(t-4)(t-2)} & -\frac{10}{t-4} & 1 
\end{bmatrix},
\]
\[
 Z_6= 
\begin{bmatrix}
1 & - \frac{15}{t-5} & \frac{45}{(t-5)(t-3)} & -\frac{15}{(t-5)(t-3)(t-1)} & 0 & 0 & 0 \\
0 & \frac{t}{t-5} & - \frac{10 t}{(t-5)(t-3)} & \frac{15t}{(t-5)(t-3)(t-1)}  & 0 & 0 & 0 \\
0 & -\frac{1}{t-5} & \frac{t^2+9}{(t-5)(t-3)} & - \frac{3(2t^2+3)}{(t-5)(t-3)(t-1)} & \frac{3}{(t-5)(t-3)} & 0 & 0\\
0 &  0 & -\frac{3 t}{(t-5)(t-3)} &  \frac{t(t^2+14)}{(t-5)(t-3)(t-1)} & - \frac{3t}{(t-5)(t-3)} & 0 & 0\\
0 & 0 & \frac{3}{(t-5)(t-3)} & - \frac{3(2t^2+3)}{(t-5)(t-3)(t-1)} &  \frac{t^2+9}{(t-5)(t-3)} & -\frac{1}{t-5}  & 0 \\
0 & 0 & 0 & \frac{15t}{(t-5)(t-3)(t-1)} & - \frac{10t}{(t-5)(t-3)} & \frac{t}{t-5} & 0 \\
0 & 0 & 0 & - \frac{15}{(t-5)(t-3)(t-1)} & \frac{45}{(t-5)(t-3)} & - \frac{15}{t-5} & 1 
\end{bmatrix}, 
\]
\[ 
 Z_7= 
\begin{bmatrix}
1 & - \frac{21}{t-6} & \frac{105}{(t-6)(t-4)} & -\frac{105}{(t-6)(t-4)(t-2)} & 0 & 0 & 0 & 0 \\
0 & \frac{t}{t-6} & - \frac{15 t}{(t-6)(t-4)} & \frac{45t}{(t-6)(t-4)(t-2)}  & -\frac{15}{(t-6)(t-4)(t-2)}   & 0 & 0 & 0 \\
0 & -\frac{1}{t-6} & \frac{t^2+14}{(t-6)(t-4)} & - \frac{5(2t^2+7)}{(t-6)(t-4)(t-2)} & \frac{15t}{(t-6)(t-4)(t-2)} & 0 & 0 &0 \\
0 &  0 & -\frac{3 t}{(t-6)(t-4)} &  \frac{t(t^2+26)}{(t-6)(t-4)(t-2)} & - \frac{3(2t^2+7)}{(t-6)(t-4)(t-2)} & \frac{3}{(t-6)(t-4)} & 0 & 0 \\
0 & 0 & \frac{3}{(t-6)(t-4)} & - \frac{3(2t^2+7)}{(t-6)(t-4)(t-2)} &  \frac{t(t^2+26)}{(t-6)(t-4)(t-2)} & -\frac{3t}{(t-6)(t-4)}  & 0 & 0 \\
0 & 0 & 0 & \frac{15t}{(t-6)(t-4)(t-2)} & - \frac{5(2t^2+7)}{(t-6)(t-4)(t-2)} & \frac{t^2+14}{(t-6)(t-4)} & -\frac{1}{t-6} & 0 \\
0 & 0 & 0 & - \frac{15}{(t-6)(t-4)(t-2)} & \frac{45t}{(t-6)(t-4)(t-2)} & - \frac{15t}{(t-6)(t-4)} & \frac{t}{t-6} & 0 \\
0 & 0 & 0 & 0 & -\frac{105}{(t-6)(t-4)(t-2)} & \frac{105}{(t-6)(t-4)} & -\frac{21}{t-6} & 1
\end{bmatrix}
.\]
\end{example}

By direct calculation, we have 
\begin{proposition}[$\det Z_d$ for $1 \le d \le 7$]
\[ \det Z_1= 1, \quad \det Z_2= \frac{t}{t-1}, \quad \det Z_3 = \frac{t^2-1}{(t-2)^2}, \quad
\det Z_4 = \frac{t^2(t^2-4)}{(t-1)^1(t-3)^3},
\]
\[ \det Z_5 = \frac{(t^2-1)^2(t^2-9)^1}{(t-2)^2(t-4)^4}, \quad
\det Z_6 = \frac{t^3(t^2-4)^2(t^2-16)^1}{(t-1)^1(t-3)^3(t-5)^5}, \]
\[
\det Z_7 = \frac{(t^2-1)^3(t^2-9)^2(t^2-25)^1}{(t-2)^2(t-4)^4(t-6)^6}.  
\]
\end{proposition}   
Since $\det X_d = \varepsilon^{d(d+1)/2} \det Z_d$, we have the following
conjecture on $ \det X_d $. 

\begin{conjecture} \label{con:det}
For a natural number $q$, we conjecture that
\begin{align*}
\det X_{2q-1} &= \bigl( \sgn(m_0) \sqrt{-1} \bigr)^{q(2q-1)} \, 
\frac{\prod\limits_{l=1}^{q-1} \bigl( t^2-(2l-1)^2 \bigr)^{q-l}}{\prod\limits_{l=1}^{q-1}(t-2l)^{2l}},
\\
\det X_{2q} &= \bigl( \sgn(m_0) \sqrt{-1} \bigr)^{q(2q+1)} \, 
\frac{t ^q\prod\limits_{l=1}^{q-1} \bigl( t^2-(2l)^2 \bigr)^{q-l}}{\prod\limits_{l=1}^{q}(t-2l+1)^{2l-1}}.
\end{align*}
\end{conjecture}

In this article, we have shown that the Siegel-Whittaker functions of rapid decay
for the large discrete series representations of $\mathrm{Sp}(2,\R)$
are described by the partially confluent hypergeometric functions $\{ f_k(a) \}_{k=0}^{d}$.
Conversely, if the above conjecture is true, we have,

\begin{corollary}
We assume that 
\[ |m_0| \ge d   \quad   \mbox{  and  } \quad  |m_0|\equiv d \pmod{2}.
\]
Let $\{ c_j (a)\}_{j=0}^{d}$ be a smooth 
solution of the system (\ref{M1}), (\ref{M2}) and (\ref{M3}) in Proposition \ref{Miyazaki}, 
and suppose that all of $\{ e^{-2 \pi(h_1a_1^2+h_2a_2^2)} c_j (a) \}_{j=0}^{d}$
are rapidly decreasing.

If Conjecture \ref{con:det} is true, 
then all of $\{ f_{k}(a) \}_{k=0}^{d}$ are $\C$-linear combinations of 
elements of $\{ c_j(a) \}_{j=0}^{d}$.
\end{corollary}
\begin{proof}
Since $|m_0| \ge d$ and $|m_0|\equiv d \pmod{2}$, we can check that
\[ \det X_d \ne 0, \]
and $X_d^{-1}$ exists. It completes the proof.
\end{proof}

\bibliographystyle{plain}
\def\cprime{$'$} \def\polhk#1{\setbox0=\hbox{#1}{\ooalign{\hidewidth
  \lower1.5ex\hbox{`}\hidewidth\crcr\unhbox0}}}
\providecommand{\bysame}{\leavevmode\hbox
to3em{\hrulefill}\thinspace}
\providecommand{\MR}{\relax\ifhmode\unskip\space\fi MR }
\providecommand{\MRhref}[2]{%
  \href{http://www.ams.org/mathscinet-getitem?mr=#1}{#2}
} \providecommand{\href}[2]{#2}

\end{document}